\documentclass{ws-ijwmip}
\usepackage{graphicx}
\usepackage[super]{cite}
\usepackage[utf8]{inputenc}

\usepackage{amsmath,amssymb,color}
\usepackage{enumerate}
\usepackage{caption,graphicx,epstopdf}
\usepackage[labelformat=simple]{subcaption}

\def\be{\begin{equation}}
\def\ee{\end{equation}}
\def\beaN{\setlength{\arraycolsep}{0.0em}\begin{eqnarray*}}
\def\eeaN{\end{eqnarray*}\setlength{\arraycolsep}{5pt}}
\def\bea{\setlength{\arraycolsep}{0.0em}\begin{eqnarray}}
\def\eea{\end{eqnarray}\setlength{\arraycolsep}{5pt}}
\def\dm{n}

\def\Rd{\RR^\dm}
\def\Zd{\ZZ^\dm}

\def\Td{\TT^\dm}
\def\CC{\mathbb{C}}
\def\NN{\mathbb{N}}
\def\RR{\mathbb{R}}
\def\ZZ{\mathbb{Z}}
\def\TT{\mathbb{T}}

\def\least{_\downarrow}
\def\up{_\uparrow}

\begin{document}

\markboth{Y. Hur}{Tight Wavelet Filter Banks with Prescribed Directions}


\title{Tight Wavelet Filter Banks with Prescribed Directions}

\author{YOUNGMI HUR}

\address{Department of Mathematics, Yonsei University\\
Seoul 03722, Korea.\footnote{Part of the work was performed while the author was visiting the Korea Institute for Advanced Study, Seoul, Korea.}\\
yhur@yonsei.ac.kr}

\maketitle

\begin{history}
\end{history}

\begin{abstract}
Constructing tight wavelet filter banks with prescribed directions is challenging.
This paper presents a systematic method for designing a tight wavelet filter bank, given any prescribed directions.
There are two types of wavelet filters in our tight wavelet filter bank. One type is entirely determined by the prescribed information about the directionality and makes the wavelet filter bank directional. The other type helps the wavelet filter bank to be tight.
In addition to the flexibility in choosing the directions, our construction method has other useful properties. 
It works for any multi-dimension, and it allows the user to have any prescribed number of vanishing moments along the chosen directions. 
Furthermore, our tight wavelet filter banks have fast algorithms for analysis and synthesis. 
Concrete examples are given to illustrate our construction method and properties of resulting tight wavelet filter banks.
\end{abstract}



\section{Introduction}
Wavelets have been studied extensively both in theory and applications \cite{Mallat, Daub, Meyer}.
In particular, tight wavelet frames for \(L_2(\Rd)\) have been used as an excellent alternative to orthonormal wavelets for \(L_2(\Rd)\). They preserve many desired properties of orthonormal wavelets while providing much more flexibility in construction. This added flexibility is proven to be helpful, especially in image processing applications. To name only a few, see, for example, Ref.~\refcite{HanQuinc,ZSh} and references therein. It is well-known that there is a well-understood method to obtain tight wavelet frames for \(L_2(\Rd)\) from tight wavelet filter banks (cf. Section~\ref{subS:tight}). 

When constructing tight wavelet filter banks, one of the sought-for properties is the directionality of the wavelet filters. 
Many kinds of research have been performed in this direction, but many of them have drawbacks limiting their use in applications. For example, the directional filter banks designed in Ref.~\refcite{DoVetterli1,VLVetterliD} are for the 2-D case only and are not tight in general. Directional Haar tight wavelet filter banks are initially developed for the 2-D case to improve the reconstruction algorithm in parallel magnetic resonance imaging (pMRI) in Ref.~\refcite{LCSHT}. This construction method is later extended to all higher dimensions in Ref.~\refcite{HanLiZhuang}, but it does not allow the user to choose the directions before entering the construction. Similar types of directional tight wavelet filter banks under a more general setting, including the case of all lowpass filters with nonnegative coefficients, are constructed in Ref.~\refcite{DiaoHan}. Construction of various complex tight wavelet filter banks exhibiting directionality is also available in the literature~\cite{Han12,Han13,HanZhao,HanMZZ}. However, one still cannot choose the directions a priori in these constructions. As a result, constructing tight wavelet filter banks with prescribed directions remains challenging.

This paper presents a new method for designing a tight wavelet filter bank, given any prescribed directions. Our construction method is based on a recent tight wavelet frame construction approach~\cite{LaiStock,HL2} that uses the trigonometric polynomial sum of hermitian squares representation. In addition to allowing the user to choose the directions, our construction method has many valuable properties, including the following:
\begin{itemlist}
\item it works for any dimension \(\dm\ge 2\);
\item it allows the user to choose the number of vanishing moments along the chosen directions;
\item the resulting tight wavelet filter banks have fast algorithms for analysis (i.e., computing wavelet coefficients) and synthesis (i.e., getting back the input).
\end{itemlist}

Despite these useful properties, our proposed construction method has some limitations. First, our tight wavelet filter bank contains wavelet filters with no directionality at all (cf.~(\ref{eq:defqCl})). Furthermore, the directionality of the directional wavelet filter is somewhat limited as it has a non-directional factor as well (cf.~(\ref{eq:defqDl})). Overcoming these limitations is future work.

The outline of the paper is as follows. Brief reviews on tight wavelet filter banks, including their connection to tight wavelet frames and recent construction methods using a trigonometric polynomial sum of hermitian squares representation, are given in Section~\ref{S:Prelim}. Section~\ref{S:Main} presents our main results about the proposed tight wavelet filter banks, with our construction method in Section~\ref{subS:Theory} and the associated fast algorithms in Section~\ref{subS:Algorithms}. Section~\ref{S:Examples} offers some examples illustrating our method and properties of the resulting tight wavelet filter banks.

\section{Preliminaries}
\label{S:Prelim} 
\subsection{Tight wavelet filter banks}
\label{subS:tight}
A {\it filter} \(h:\Zd\to\RR\) is a function defined on \(\dm\)-dimensional integer grids. 
Although some of the results in this paper, including our main construction theorem (i.e., Theorem~\ref{thm:main}), still hold for general dilation matrices, their interpretation in connection with directional filters could be unclear. Hence, we consider in this paper an integer {\it dilation (factor)} \(\lambda\ge 2\) only.  We say a filter is a {\it lowpass} or {\it refinement} filter if \(\sum_{k\in\Zd}h(k)=\lambda^{\dm/2}\), and a {\it highpass} or {\it wavelet} filter if \(\sum_{k\in\Zd}h(k)=0\). 

A trigonometric polynomial \(\tau:\Td\to\CC\) is called a {\it mask} (associated with the filter \(h\)) if  
\be
\label{eq:deftauh}
\tau(\omega)=\sum_{k\in\Zd} h(k) e^{-ik\cdot\omega}.
\ee
In this paper, any property associated with the filter will be related to the mask, and vice versa, without explicit mention.
For example, a {\it lowpass} mask \(\tau\) is a mask with \(\tau(0)=\lambda^{\dm/2}\).

For a dilation factor \(\lambda\), let \(\Lambda\) be a complete set of representatives of the distinct cosets of the quotient
group \(\Zd/\lambda\Zd\) containing \(0\), and let \(\Gamma\) be a complete
set of representatives of the distinct cosets of
\(2\pi((\lambda^{-1}\Zd)/\Zd)\) containing \(0\). 

For a nonnegative integer \(m\), \(\tau\) has {\it accuracy number} \(m\) if the minimum order of the roots that \(\tau\) has at the point in \(\Gamma\setminus\{0\}\) is \(m\). The {\it flatness number} of \(\tau\) is the order of the root that \(\tau-\lambda^{\dm/2}\) has at \(0\). In particular, a lowpass mask \(\tau\) always has positive flatness, and it has positive accuracy if and only if \(\tau(\gamma)=0\), for every \(\gamma\in\Gamma\setminus\{0\}\). The order of the root that \(\tau\) has at \(0\) is called its {\it number of vanishing moments}. Hence any wavelet mask has at least one vanishing moment.

Let \(\tau\) be a lowpass mask with dilation \(\lambda\). We recall that \(\phi\in L_2(\Rd)\) is a {\it refinable function} (associated with \(\tau\)) if it satisfies \(\widehat{\phi}(\lambda\omega)=\lambda^{-\dm/2}\tau(\omega)\widehat{\phi}(\omega)\), \(\omega\in\Rd\). Here and below, \(\widehat{f}\) is the Fourier transform of \(f\in L_{2}(\Rd)\), defined as, for \(f\in L_{1}(\Rd)\cap L_{2}(\Rd)\),  \(\widehat{f}(\omega)=\int_{\Rd}f(x)e^{-ix\cdot\omega}\mathrm{d}x,\) \(\omega\in\Rd\).

Given the wavelet masks \(q_{i}\), \(1\leq i\leq r,\) \(\psi_{i}\in L_2(\Rd)\) is the {\it mother wavelet} if it satisfies \(\widehat{\psi_i}(\lambda\omega)=\lambda^{-\dm/2}q_{i}(\omega)\widehat{\phi}(\omega),\,\omega\in\Rd\), and 
\[
X(\{\psi_1,\ldots,\psi_r\}):=\{\lambda^{jn/2}\psi_i(\lambda^j\cdot\,-k):1\leq i\leq r; j\in\ZZ,k\in\Zd\}
\]
is the {\it wavelet system} generated by \(\psi_{i}\), \(1\leq i\leq r\).
The wavelet system is called an {\it (MRA-based) tight wavelet frame} if it forms a tight frame for \(L_2(\Rd)\).

Tight wavelet frames can be obtained by using a method called the {\it unitary extension principle} (UEP) \cite{HanOrtho,RonShen}. The following statement is taken from Ref.~\refcite{HanOrtho} but adapted to the setting of our paper.

\begin{theorem}[Ref.~\refcite{HanOrtho}]
\label{thm:uep}
Let \(\tau\) be a lowpass mask with dilation \(\lambda\), and let \(\phi\) be a compactly supported distribution defined by \(\widehat{\phi}(\omega):=\prod_{j=1}^{\infty}\lambda^{-\dm/2}\tau(\lambda^{-j}\omega)\). If \(q_{i},\,1\leq i\leq r\), are trigonometric polynomials such that for all \(\omega\in\Td\) and for all \(\gamma\in \Gamma\):
\begin{equation}
\label{eq:uepcond}
\tau(\omega)\overline{\tau(\omega+\gamma)}+\sum_{i=1}^{r}q_{i}(\omega)\overline{q_{i}(\omega+\gamma)}
=\begin{cases} \lambda^\dm,&\gamma=0\\0,&\gamma\ne 0,\end{cases}
\end{equation}
then \(\phi\) must be a function in \(L_2(\Rd)\) and the wavelet system \(X(\{\psi_1,\ldots,\psi_r\})\) generated by \(\psi_{i}\), \(1\leq i\leq r\), defined by \(\widehat{\psi_i}(\lambda\omega)=\lambda^{-\dm/2}q_{i}(\omega)\widehat{\phi}(\omega)\), \(\omega\in\Rd\), is a tight wavelet frame for \(L_{2}(\Rd)\).
 \end{theorem}

We shall call the set \(\{\tau, q_1, \dots, q_r\}\) in the above theorem a {\it tight wavelet filter bank}. In particular, the masks \(q_{i},\,1\leq i\leq r\) in the tight wavelet filter bank \(\{\tau, q_1, \dots, q_r\}\) are wavelet masks.

\subsection{Tight wavelet filter banks from sos representations}
\label{subS:SOS}
The polynomial sum of squares representation has been extensively studied in the literature \cite{Art,Pfister} but recently has been connected with tight wavelet filter bank construction \cite{LaiStock,ScheidStock,ScheidStockII,HL1,HL2,HLO}. In these approaches, given a lowpass mask \(\tau\), one tries to find a sum of hermitian squares representation of a trigonometric polynomial derived from \(\tau\).

For a given nonnegative trigonometric polynomial \(f\),  \(f\) is said to have a {\it sum of hermitian squares (SOS) representation} if there exist trigonometric polynomials \(g_1,\ldots,g_N\) such that 
\[
f(\omega)=\sum_{l=1}^N|g_l(\omega)|^2,\quad \forall \omega\in\Td.
\]
A key observation in the sos-based tight wavelet filter bank construction is that, for a lowpass mask \(\tau\) with dilation \(\lambda\), finding an sos of the trigonometric polynomial
\[1-\lambda^{-\dm}\sum_{\gamma\in\Gamma} |\tau(\omega/\lambda+\gamma)|^2\]
is sufficient to obtain wavelet masks \(q_1, \dots, q_r\) such that \(\{\tau, q_1, \dots, q_r\}\) is a tight wavelet filter bank. 
The precise statement is given below, which is taken from Ref.~\refcite{LaiStock} but adapted to the notation of this paper. 

\begin{theorem}[Ref.~\refcite{LaiStock}]
\label{thm:sostwf}
Let \(\tau\) be a lowpass mask with dilation \(\lambda\), satisfying 
\[
1-\lambda^{-\dm}\sum_{\gamma\in\Gamma} |\tau(\omega/\lambda+\gamma)|^2=\sum_{l=1}^{N}|g_l(\omega)|^{2}, \quad \omega\in\Td
\]
for some trigonometric polynomials \(g_l\), \(1\le l \le N<+\infty\). 
Let \(h\) be the lowpass filter corresponding to \(\tau\), and for each \(\nu\in\Lambda\), define \(\tau_\nu\) by 
\be
\label{eq:deftaunu}
\tau_\nu(\omega) = \sum_{k\in\Zd} h(\lambda k-\nu) e^{-ik\cdot\omega},\quad\omega\in\Td.
\ee
Then, with the lowpass mask \(\tau\), the \(N+\lambda^\dm\) functions
\be
\label{eq:defq1}
q_{1,l}(\omega)=\tau(\omega){g}_{l}(\lambda\omega),\quad l=1,\ldots,N,\quad  \omega\in\Td
\ee
\be
\label{eq:defq2}
q_{2,\nu}(\omega)= e^{i\nu\cdot\omega}-\tau(\omega)\overline{\tau_{\nu}(\lambda\omega)}, \quad\nu\in\Lambda, \quad \omega\in\Td
\ee
form a tight wavelet filter bank.
\end{theorem}

{\noindent\bf Remark:}
\(\tau_\nu\) in (\ref{eq:deftaunu}) is the {\it polyphase component} of \(\tau\), associated with \(\nu\in\Lambda\)~\cite{LaiStock,HL1,HL2,HLO}. Since we believe providing this connection in detail may make the presentation of this paper confusing, we only mention it here without further explanations.
\qed

\medskip
When tight wavelet filter banks are constructed using the method in the above theorem, the number of vanishing moments of wavelet masks can be written in terms of the properties of the trigonometric polynomials \(\tau\) and \(g_l\), \(l=1,\ldots,N\).

\begin{theorem}[Ref.~\refcite{HL2}]
\label{thm:vmsostwf}
Assume the settings of Theorem~\ref{thm:sostwf}. Suppose that \(\tau\) has accuracy number \(a\ge 1\), and the flatness number \(b\). Then, for each \(1\leq l\leq N\), the wavelet mask \(q_{1,l}\) in (\ref{eq:defq1}) has exactly as many vanishing moments as \(g_{l}\), and the wavelet masks \(q_{2,\nu},\nu\in\Lambda\) in (\ref{eq:defq2}) have at least \(\min\{a,b\}\) vanishing moments.
\end{theorem}

We mention in passing a related result, which is about the minimum number of vanishing moments of all wavelet masks. This number is precisely \(\min\{a,c/2\}\), with \(a\) the same as above and \(c\) the order of the root that \(|\tau|^2-\lambda^{\dm}\) has at \(0\)~\cite{Han13}.

One of the main difficulties in the sos-based construction methods such as Theorem~\ref{thm:sostwf} is that, except in the 1-D case, finding an sos representation of a given nonnegative trigonometric polynomial is nontrivial. In the 1-D case, there is no such difficulty thanks to the following Fej\'{e}r-Riesz Lemma~\cite{Daub}. 

\begin{theorem}[Ref.~\refcite{Daub}]
\label{thm:FRlemma}
Let \(g\) be a nonnegative trigonometric polynomial such that \(g(\omega)=\sum_{k=-d}^{d}c_{k}e^{-ik\omega}\), \(\omega\in\TT\), with real coefficients \(c_k\), for some nonnegative integer \(d\). Then there exists a trigonometric polynomial \(g_{1/2}(\omega)=\sum_{k=0}^{d}a_{k}e^{-ik\omega}\), \(\omega\in\TT\), with real coefficients \(a_k\), such that \(|g_{1/2}(\omega)|^{2}=g(\omega)\), \(\omega\in\TT\).
\end{theorem}

\section{Tight Wavelet Filter Banks with Prescribed Directions}
\label{S:Main}
\subsection{Construction of tight wavelet filter banks}
\label{subS:Theory}

Fix a positive integer \(N\in\NN\), the number of directions. Let \(\zeta_1, \ldots,\zeta_N\in \Zd\) be the initial points, 
and let \(\eta_1, \ldots,\eta_N\in \Zd\) be the terminal points of the \(N\) directions.
For each \(l=1,\dots,N\), choose a positive integer \(m_l\in\NN\), the number of vanishing moments along the \(\xi_l:=\eta_l-\zeta_l\) direction.
Then we will show that the directional wavelet mask of our tight wavelet filter bank has the factor (cf. Remark~2 after Theorem~\ref{thm:main} and (\ref{eq:defqDl}))
\[
(e^{-i\zeta_l\cdot\omega}-e^{-i\eta_l\cdot\omega})^{m_l}=e^{-im_l\zeta_l\cdot\omega}(1-e^{-i\xi_l\cdot\omega})^{m_l},
\]
up to scale. Before presenting the main result of this paper, we set some terminology. We refer to \([\xi_1, \dots,\xi_N]\in \ZZ^{\dm\times N}\) as the {\it direction matrix} and \((m_1, \dots, m_N)\) as the {\it vanishing moment vector}.

\begin{theorem}
\label{thm:main}
For \(N\ge 1\), let \([\xi_1, \dots,\xi_N]\) and \((m_1, \dots, m_N)\) be a direction matrix and a vanishing moment vector, respectively.
Let \(\lambda\ge 2\) be an integer such that \(N\le \lambda^\dm\), and let \(\Lambda=\{\nu_1,\dots,\nu_{\lambda^\dm}\}\) be a set of representatives of the distinct cosets of \(\Zd/\lambda\Zd\) containing \(0\).
Then there exist trigonometric polynomials \(p_l\), \(1\le l\le N\), such that the lowpass mask \(\tau\) defined as 
\be
\label{eq:deftau}
\tau(\omega)=\sum_{l=1}^{N}\lambda^{-\dm/2}p_{l}(\lambda\omega)e^{i\nu_l\cdot\omega}+ \sum_{l=N+1}^{\lambda^\dm}\lambda^{-\dm/2}e^{i\nu_l\cdot\omega},
\ee
and the wavelet masks \(q_{D,l}\), \(1\le l\le N\), and \(q_{C,\mu}\), \(1\le \mu\le \lambda^\dm\), defined as
\be
\label{eq:defqDl}
q_{D,l}(\omega)=\lambda^{-\dm/2}2^{-m_l}(1-e^{-i\lambda \xi_l\cdot\omega})^{m_l}\tau(\omega),
\ee
\be
\label{eq:defqCl}
q_{C,\mu}(\omega)=
\begin{cases} 
e^{i\nu_\mu\cdot\omega}-\lambda^{-\dm/2}\tau(\omega)\overline{p_\mu(\lambda\omega)}&1\le \mu\le N,\\
e^{i\nu_\mu\cdot\omega}-\lambda^{-\dm/2}\tau(\omega)&N+1\le \mu\le \lambda^\dm,
\end{cases}
\ee
form a tight wavelet filter bank. 
\end{theorem}
{\noindent\bf Remark:} If \(N=\lambda^\dm\), there is no index \(l\) such that \(N+1\le l \le \lambda^\dm\) in the definition of \(\tau\) in (\ref{eq:deftau}), hence in this case, \(\tau\) is defined as \(\tau(\omega)=\sum_{l=1}^{N}\lambda^{-\dm/2}p_{l}(\lambda\omega)e^{i\nu_l\cdot\omega}\).
\quad \qed

\begin{proof}
For \(1\le l\le N\), let 
\be
\label{eq:defgl}
g_l(\omega)=\lambda^{-\dm/2}2^{-m_l}(1-e^{-i\xi_l\cdot\omega})^{m_l}.
\ee 
Then, since \(|g_l(\omega)|^2=\lambda^{-\dm}\sin^{2m_l}\left(\xi_l\cdot\omega/2\right)\), we have
\[1-\sum_{l=1}^{N}|g_l(\omega)|^2=\sum_{l=1}^{N}\frac{1}{\lambda^{\dm}}\left(1-\sin^{2m_l}\left(\frac{\xi_l\cdot\omega}{2}\right)\right)+\sum_{l=N+1}^{\lambda^\dm}\frac{1}{\lambda^{\dm}},\quad \forall \omega\in\Td.\]
Since  \(1-\sin^{2m_l}(t)\ge 0\) for all \(t\in\TT\), by Theorem~\ref{thm:FRlemma}, there exists a univariate trigonometric polynomial \(b_{m_l,1/2}\) such that \(1-\sin^{2m_l}(t)=|b_{m_l,1/2}(t)|^2\), \(t\in\TT\).
Note that we can choose \(b_{m_l,1/2}\) to satisfy \(b_{m_l,1/2}(0)=1\).
Let \(p_l(\omega):=b_{m_l,1/2}(\xi_l\cdot\omega)\), \(\omega\in\Td\). Then \(p_l(0)=1\) and
\be
\label{eq:oneminussumgl}
1-\sum_{l=1}^{N}|g_l(\omega)|^2=\sum_{l=1}^{N}\left|\frac{1}{\lambda^{\dm/2}}p_l(\omega)\right|^2+\sum_{l=N+1}^{\lambda^\dm}\left|\frac{1}{\lambda^{\dm/2}}\right|^2.
\ee
Define \(\tau\) as in (\ref{eq:deftau}). 
Then \(\tau\) is a lowpass mask since \(\tau(0)=\lambda^{\dm/2}\), and we have
\beaN
\lambda^{\dm}\sum_{\gamma\in\Gamma} \left|\tau\left(\frac{\omega}{\lambda}+\gamma\right)\right|^2
&=&\sum_{\gamma\in\Gamma}\left(\sum_{l=1}^{N}p_{l}\left(\lambda\left(\frac{\omega}{\lambda}+\gamma\right)\right)e^{i\nu_l\cdot\left(\frac{\omega}{\lambda}+\gamma\right)}+ \sum_{l=N+1}^{\lambda^\dm}e^{i\nu_l\cdot\left(\frac{\omega}{\lambda}+\gamma\right)}\right)\\
&\cdot&\left(\sum_{l'=1}^{N}\overline{p_{l'}\left(\lambda\left(\frac{\omega}{\lambda}+\gamma\right)\right)}e^{-i\nu_{l'}\cdot\left(\frac{\omega}{\lambda}+\gamma\right)}+ \sum_{l'=N+1}^{\lambda^\dm}e^{-i\nu_{l'}\cdot\left(\frac{\omega}{\lambda}+\gamma\right)}\right).
\eeaN
Since \(p_l\) is a trigonometric polynomial and \(\lambda\gamma\in 2\pi\Zd\), the above expression becomes
\beaN
&&\sum_{l=1}^{N}\;\sum_{l'=1}^{N}\left(\sum_{\gamma\in\Gamma}e^{i(\nu_l-\nu_{l'})\cdot\gamma}\right)p_{l}(\omega)\overline{p_{l'}(\omega)}e^{i(\nu_l-\nu_{l'})\cdot\omega/\lambda}\\
&+&\sum_{l=N+1}^{\lambda^\dm}\;\sum_{l'=1}^{N}\left(\sum_{\gamma\in\Gamma}e^{i(\nu_l-\nu_{l'})\cdot\gamma}\right)\overline{p_{l'}(\omega)}e^{i(\nu_l-\nu_{l'})\cdot\omega/\lambda}\\
&+&\sum_{l=1}^{N}\;\sum_{l'=N+1}^{\lambda^\dm}\left(\sum_{\gamma\in\Gamma}e^{i(\nu_l-\nu_{l'})\cdot\gamma}\right)p_{l}(\omega)e^{i(\nu_l-\nu_{l'})\cdot\omega/\lambda}\\
&+&\sum_{l=N+1}^{\lambda^\dm}\;\sum_{l'=N+1}^{\lambda^\dm}\left(\sum_{\gamma\in\Gamma}e^{i(\nu_l-\nu_{l'})\cdot\gamma}\right)e^{i(\nu_l-\nu_{l'})\cdot\omega/\lambda}=\sum_{l=1}^{N}|p_l(\omega)|^2+\sum_{l=N+1}^{\lambda^\dm}1,
\eeaN
where, for the equality, the following known identity (cf. Ref.~\refcite{ECLP})
\[
\sum_{\gamma\in\Gamma}e^{i(\nu_l-\nu_{l'})\cdot\gamma}=
\begin{cases} 
\lambda^\dm&\nu_l=\nu_{l'},\\
0&\nu_l\ne\nu_{l'},
\end{cases}
\]
is used. Thus, from (\ref{eq:oneminussumgl}), we get
\[
1-\lambda^{-\dm}\sum_{\gamma\in\Gamma} |\tau(\omega/\lambda+\gamma)|^2=\sum_{l=1}^N|g_l(\omega)|^2.
\]
If \(h\) is the filter corresponding to the lowpass mask \(\tau\), then 
\[
\tau(\omega)=\sum_{k\in\Zd}\sum_{l=1}^{\lambda^\dm}h(\lambda k-\nu_l) e^{-i(\lambda k-\nu_l)\cdot\omega}=\sum_{l=1}^{\lambda^\dm}\left(\sum_{k\in\Zd}h(\lambda k-\nu_l) e^{-i(k\cdot(\lambda\omega))}\right)e^{i\nu_l\cdot\omega}.\]
By comparing this last expression with (\ref{eq:deftau}), we see that
\[
\sum_{k\in\Zd} h(\lambda k-{\nu_l}) e^{-ik\cdot\omega}=\begin{cases} 
\lambda^{-\dm/2}p_l(\omega)&1\le l\le N\\
\lambda^{-\dm/2}& N+1\le l\le \lambda^\dm.
\end{cases}
\]
Define \(q_{D,l}\), \(1\le l\le N\) and \(q_{C,\mu}\), \(1\le \mu\le \lambda^\dm\) as in (\ref{eq:defqDl}) and (\ref{eq:defqCl}), respectively.
Then, by Theorem~\ref{thm:sostwf}, these \(N+\lambda^\dm\) functions with the lowpass mask \(\tau\) form a tight wavelet filter bank.
\end{proof}

{\noindent\bf Remark 1:}
We refer to the tight wavelet filter bank constructed in Theorem~\ref{thm:main} as the {\it tight wavelet filter bank with prescribed directions (TWFPD)}.
We refer to \(q_{D,l}\), \(1\le l\le N\) as the {\it directional wavelet masks}, and \(q_{C,\mu}\), \(1\le \mu\le \lambda^\dm\) as the {\it complementary wavelet masks} of TWFPD.
\qed

\medskip

{\noindent\bf Remark 2:}
In Theorem~\ref{thm:main} and the TWFPD,  with the initial point \(\zeta_l\) of direction \(\xi_l\), one can replace (\ref{eq:defqDl}) and (\ref{eq:defgl}) by the directional wavelet mask \(q_{D,l}\) of the form
\[
q_{D,l}(\omega)=\tau(\omega)g_l(\lambda \omega)=\lambda^{-\dm/2}2^{-m_l}e^{-i\lambda m_l\zeta_l\cdot\omega}(1-e^{-i\lambda \xi_l\cdot\omega})^{m_l}\tau(\omega)
\]
and the function \(g_l\) of the form
\be
\label{eq:defglg}
g_l(\omega)=\lambda^{-\dm/2}2^{-m_l}e^{-im_l\zeta_l\cdot\omega}(1-e^{-i\xi_l\cdot\omega})^{m_l},
\ee 
respectively. Since the two different forms of the directional wavelet mask \(q_{D,l}\) differ only by the shift of the corresponding mother wavelet, the constructed tight wavelet frames are the same. 
\qed

%
%

\medskip

Since any MRA-based wavelet system has fast algorithms, our tight wavelet frame associated with TWFPD has fast algorithms. Because we have a tight wavelet frame, the same filters are used both for the analysis algorithm and the synthesis algorithm. In addition to these standard fast algorithms, because of the specific way that our TWFPD is constructed, an alternative synthesis algorithm is available for our TWFPD. More details about these algorithms are given in Section~\ref{subS:Algorithms}.

The TWFPD construction in Theorem~\ref{thm:main} is a particular case of the sos-based construction of tight wavelet filter banks given in Theorem~\ref{thm:sostwf}. The following corollary is an immediate consequence of Theorem~\ref{thm:vmsostwf}  when applied to our TWFPD wavelet masks. 

\begin{corollary}
\label{coro:VM}
Let \(\{\tau,q_{D,1},\ldots,q_{D,N}, q_{C,1},\ldots, q_{C,\lambda^\dm}\}\) be the TWFPD constructed in Theorem~\ref{thm:main} with the vanishing moment vector \((m_1, \dots, m_N)\).
Then for each \(l=1,\dots, N\), the directional wavelet mask \(q_{D,l}\) has exactly \(m_l\) vanishing moments, and for each \(\mu=1,\dots, \lambda^\dm\), the complementary wavelet mask \(q_{C,\mu}\) has at least \(\min\{a,b\}\ge 1\) vanishing moments, where \(a\) is the accuracy number and \(b\) is the flatness number of \(\tau\). 
\end{corollary}

\begin{proof}
We recall \(|g_l(\omega)|^2=\lambda^{-\dm}\sin^{2m_l}\left(\xi_l\cdot\omega/2\right),\)
from the proof of Theorem~\ref{thm:main}, where \(g_l\) is defined as in (\ref{eq:defgl}). Hence \(g_l\) has exactly \(m_l\) vanishing moments, which implies that the directional wavelet mask \(q_{D,l}\) has exactly \(m_l\) vanishing moments, since \(q_{D,l}(\omega)=\tau(\omega)g_{l}(\lambda\omega)\), \(\omega\in\Td\).

To estimate the number of vanishing moments for complementary wavelet masks, we first note that \(b\ge 1\) since \(\tau\) of TWFPD is a lowpass mask. Also, \(\tau\) has positive accuracy (i.e. \(a\ge 1\)) since, for any \(\gamma\in\Gamma\setminus\{0\}\), 
\[\tau(\gamma)=\sum_{l=1}^{N}\lambda^{-\dm/2}p_{l}(\lambda\gamma)e^{i\nu_l\cdot\gamma}+ \sum_{l=N+1}^{\lambda^\dm}\lambda^{-\dm/2}e^{i\nu_l\cdot\gamma}=\lambda^{-\dm/2}\sum_{l=1}^{\lambda^\dm}e^{i\nu_l\cdot\gamma}=0.\]
where, the facts that \(\lambda\gamma\in 2\pi\Zd\), \(p_l\) is a trigonometric polynomial, and \(p_l(0)=1\), and the identity (cf. Ref.~\refcite{ECLP}) \(
\sum_{l=1}^{\lambda^\dm}e^{i\nu_l\cdot\gamma}=0,
\)
\(\gamma\in\Gamma\setminus\{0\}\) are used.
Thus, we have \(\min\{a,b\}\ge 1\). By invoking Theorem~\ref{thm:vmsostwf}, we see that each complementary wavelet mask \(q_{C,\mu}\) has at least \(\min\{a,b\}\ge 1\) vanishing moments.
\end{proof}

From Theorem~\ref{thm:main} and Corollary~\ref{coro:VM}, we see that the two types of wavelet masks in TWFPD, the directional ones and the complementary ones, play quite a different role. 

Each directional wavelet mask \(q_{D,l}\) has \(g_l(\lambda\cdot)\) as a factor, entirely determined from the input: the directionality vector \(\xi_l\), the vanishing moment number \(m_l\) along the \(\xi_l\) direction, and the number \(\lambda\) with \(N\le \lambda^\dm\), where \(N\) is the number of directions. The number of vanishing moments of \(q_{D,l}\) is given exactly as \(m_l\). 

On the other hand, the primary role of the complementary wavelet masks in TWFPD is to complement the directional wavelet masks so that when combined, they form a tight wavelet filter bank. 

Additionally, these complementary wavelet masks allow TWFPD to have an alternative synthesis algorithm. More precisely, the standard fast analysis algorithm of tight wavelet filter banks using the wavelet masks \(q_{2,\nu}\), \(\nu\in\Lambda\) in Theorem~\ref{thm:sostwf} generates the detail coefficients in the analysis part of Laplacian pyramid (LP) algorithms \cite{BA}. This connection is easy to observe and can be found, for example, in Ref.~\refcite{HLO}. It is well known that the LP algorithms have a simple reverse process of the analysis algorithm as a synthesis algorithm~\cite{DoVetterli2,ECLP}. Because our complementary wavelet masks \(q_{C,\mu}\) correspond to the wavelet masks \(q_{2,\nu}\), \(\nu\in\Lambda\) in Theorem~\ref{thm:sostwf}, they provide an alternative synthesis algorithm for TWFPD, as shown in detail below.

\subsection{Fast TWFPD algorithms}
\label{subS:Algorithms}
For a dilation \(\lambda \ge 2\), we recall that the {\it downsampling} operator \(\least\) and the {\it upsampling} operator \(\up\) are defined as
\[
x\least(k):=x(\lambda k),\quad k\in\Zd,\quad 
x\up(k):=
\begin{cases} 
x(k/\lambda),&k\in\lambda\Zd\\
                  0,&k\in\Zd\setminus\lambda\Zd.
\end{cases}
\]

Let \(\{\tau,q_{D,1},\ldots,q_{D,N}, q_{C,1},\ldots, q_{C,\lambda^\dm}\}\) be the TWFPD constructed in Theorem~\ref{thm:main}.
We use \(x_j\) to denote the TWFPD {\it coarse coefficients} at level \(j\) and use \(d_{j,D,l}\), \(1\le l\le N\) and \(d_{j,C,\mu}\), \(1\le \mu\le \lambda^\dm\) to denote the TWFPD {\it detail coefficients} at level~\(j\).
The {\it standard TWFPD analysis} algorithm computes the coefficients \(x_j\), \(d_{j,D,l}\), and \(d_{j,C,\mu}\) from the coarse coefficients \(x_{j+1}\) by convolving with TWFPD filters followed by downsampling. The {\it standard TWFPD synthesis} algorithm gets back the coarse coefficients \(x_{j+1}\) by upsampling the coefficients \(x_j\), \(d_{j,D,l}\), and \(d_{j,C,\mu}\), then by convolving with TWFPD filters, and finally by summing them up. These algorithms are the standard fast algorithms for a tight wavelet filter bank but applied to our TWFPD filters. 

\medskip\noindent{\bf Standard Fast TWFPD Algorithms.}
Let \(h\) be the lowpass filter associated with \(\tau\), i.e. \(\tau(\omega)=\sum_{k\in\Zd} h(k)e^{-ik\omega}\).
Let \(h_{l}\), \(1\le l\le N\) be the highpass filters associated with \(g_l\), \(1\le l\le N\) in (\ref{eq:defgl}) (or in (\ref{eq:defglg}) with the nontrivial initial point of direction) via \(g_l(\omega)=\sum_{k\in\Zd}h_{l}(k)e^{-ik\cdot\omega}\). 
As before, let \(\Lambda\) consist of \(\{\nu_1,\dots,\nu_{\lambda^\dm}\}\), the set of representatives of the distinct cosets of \(\Zd/\lambda\Zd\) containing \(0\).
For \(1\le l\le \lambda^\dm\), define \(\delta_{l}:\Zd\to \{0,1\}\) as the filter, which always takes 0 except for \(\delta_{l}(\nu_l)=1\), where \(\nu_l\in\Lambda\).
Define \(\widetilde{h}\), \(\widetilde{h}_{l}\), \(\widetilde{\delta}_{l}\) as \(\widetilde{h}(k):=h(-k)\), \(\widetilde{h}_{l}(k):=h_{l}(-k)\), and \(\widetilde{\delta}_{l}(k):=\delta_{l}(-k)\), \(k\in \Zd\), respectively.

\medskip
{\obeylines{{\tt
{\it input}  \(x_{J+1}: \Zd\to\RR\)
\smallskip
{\bf Standard TWFPD Analysis: computing \(x_{j}\), \(d_{j,D,l}\), and \(d_{j,C,\mu}\) from \(x_{j+1}\)}
for \(j=J,J-1,\ldots,0 \)
\ \ \ \(x_{j}=(\widetilde{h}\ast x_{j+1})\least\) \ \ \ \ \ \ \ \ \ \ \ \ \ \ \ \ \ \ \ \ \ \ \ \ \;(i)
\ \ \ for \(l=1,2,\dots,N\)
\ \ \ \ \ \ \(d_{j,D,l}=\widetilde{h}_{l}\ast x_j\) \ \ \ \ \ \ \ \ \ \ \ \ \ \ \ \ \ \ \ \ \ \ \;(ii)
\ \ \ end
\ \ \ for \(\mu=1,2,\dots,\lambda^\dm\)
\ \ \ \ \ \ \(d_{j,C,\mu}=(x_{j+1}-h\ast (x_j{\up}))(\lambda\cdot-\nu_\mu)\) \ \ \ \ \ \;(iii)
\ \ \ end
end
}}}
\medskip
{\obeylines{{\tt
{\bf Standard TWFPD Synthesis: computing \(x_{j+1}\) from \(x_{j}\), \(d_{j,D,l}\), and \(d_{j,C,\mu}\)}
for \(j=0,1,\ldots,J\)
\ \ \ \(x_{j+1}=h\ast (x_j{\up})+\sum_{l=1}^N ((h_l{\up})\ast h)\ast(d_{j,D,l}{\up})\)
\ \ \ \ \ \ \ \(+\sum_{\mu=1}^{\lambda^\dm} (\widetilde{\delta}_{\mu}-\sum_{m\in\Zd} h(-\lambda m-\nu_\mu) h(\cdot-\lambda m))\ast(d_{j,C,\mu}{\up})\)
end
}}}

\medskip
The standard TWFPD synthesis algorithm given above is immediate once we observe that the highpass filters corresponding to the TWFPD wavelet masks \(q_{D,l}\) and \(q_{C,\mu}\) are \((h_l{\up})\ast h\) and \(\widetilde{\delta}_{\mu}-\sum_{m\in\Zd} h(-\lambda m-\nu_\mu) h(\cdot-\lambda m)\), respectively.
In the standard TWFPD analysis algorithm, the detail coefficients are computed easily because of the specific form that our TWFPD wavelet filters take. The coefficients in step {\tt (ii)} are obtained by observing
\[d_{j,D,l}=((\widetilde{h}_l{\up})\ast \widetilde{h}\ast x_{j+1})\least=\widetilde{h}_l\ast(\widetilde{h}\ast x_{j+1})\least=\widetilde{h}_l\ast x_j,\]
and the coefficients \(d_{j,C,\mu}\) in step {\tt (iii)} are from \((\delta_{\mu}\ast x_{j+1})\least=x_{j+1}(\lambda\cdot-\nu_\mu)\) and
\[\left(\left(\sum_{m\in\Zd} h(\lambda m-\nu_\mu) \widetilde{h}(\cdot-\lambda m)\right)\ast x_{j+1}\right)\least =h(\lambda \cdot-\nu_\mu)\ast x_j = (h\ast (x_j{\up}))(\lambda\cdot-\nu_\mu).\]

\medskip\noindent{\bf LP-based Fast TWFPD Synthesis Algorithm.}
As we discussed at the end of Section~\ref{subS:Theory}, our complementary wavelet masks can be understood in connection with the LP analysis algorithm. From this, we have a simple synthesis process to obtain \(x_{j+1}\) using the coarse coefficients \(x_{j}\) and the detail coefficients  \(d_{j,C,\mu}\) only.

\medskip
{\obeylines{{\tt
{\bf LP-based TWFPD Synthesis: computing \(x_{j+1}\) from \(x_{j}\) and \(d_{j,C,\mu}\)}
for \(j=0,1,\ldots,J\)
\vskip0.1cm
\ \ \ for \(\mu=1,2,\dots,\lambda^\dm\)
\ \ \ \ \ \ \(x_{j+1}(\lambda\cdot-\nu_\mu)=(h\ast (x_j{\up}))(\lambda\cdot-\nu_l)+d_{j,C,\mu}\)\ \ \ \ \ (iv)
\ \ \ end
end
}}}

\medskip
We see that step {\tt (iv)} is simply the reverse process of step {\tt (iii)} and that, unlike the standard TWFPD synthesis algorithm we saw earlier, the detail coefficients \(d_{j,D,\mu}\) are not used for computing \(x_{j+1}\) in this LP-based TWFPD synthesis algorithm. 

Depending on what one would like to have for synthesis algorithms, one can choose which synthesis algorithm to use. If one wants a faster synthesis algorithm, then the LP-based TWFPD synthesis algorithm can be used because it is much faster, as we will quantify soon. Suppose one wants to remove noises in the coefficients after the standard TWFPD analysis algorithm. In that case, the standard TWFPD synthesis algorithm is the one to use since it is much more effective in removing such noises \cite{DoVetterli1}.

\medskip 
Next, we discuss the complexity of TWFPD algorithms.
We measure the complexity by counting the number of multiplicative operations needed in a complete cycle of \(1\)-level-down analysis and \(1\)-level-up synthesis, meaning the number of operations required to obtain \(x_j\), \(d_{j,D,l}\), and \(d_{j,C,\mu}\) from \(x_{j+1}\) and get back \(x_{j+1}\). We consider the two fast algorithms with the same analysis algorithm, the one with the standard synthesis algorithm and another with the LP-based synthesis algorithm.

\medskip\noindent{\bf Complexity of Standard Fast TWFPD Algorithms.}
Suppose that at level \(j+1\), we have the coarse coefficients \(x_{j+1}\) with \(L\) data points. For simplicity, we assume that \(L\) is a multiple of \(\lambda^\dm\), where \(\lambda\) is the dilation factor. Then after \(1\)-level-down analysis, we obtain the coarse coefficients \(x_{j}\) with \(L/\lambda^\dm\) data points in step {\tt (i)}, the detail coefficients \(d_{j,D,l}\) in step {\tt (ii)}, and  the detail coefficients \(d_{j,C,\mu}\) in step {\tt (iii)}. We then get back \(x_{j+1}\) from the coarse coefficients \(x_{j}\) and the detail coefficients \(d_{j,D,l}\) and  \(d_{j,C,\mu}\) using \(1\)-level-up synthesis with standard TWFPD synthesis algorithm. 

Let \(\alpha\) and \(\beta_l\) be the number of nonzero entries in the lowpass filter \(h\) and the highpass filter \(h_l\). Given \(x_{j+1}\) with \(L\) data points, the number of multiplicative operations needed in a complete cycle of \(1\)-level-down analysis and \(1\)-level-up synthesis with standard TWFPD synthesis algorithm is the sum of the following numbers:
\begin{itemize}
\item \(\alpha L\) \quad  \quad \quad \quad \quad \quad  \quad \quad \quad \quad \;\; [for step {\tt (i)} in Standard TWFPD Analysis]
\item \((\sum_{l=1}^N\beta_l/\lambda^\dm)L\) \quad  \quad \quad \quad \quad  \;\, [for step {\tt (ii)} in Standard TWFPD Analysis]
\item \(\alpha L\) \quad  \quad \quad \quad \quad \quad  \quad \quad \quad \quad \quad [for step {\tt (iii)} in Standard TWFPD Analysis]
\item \((\alpha+\sum_{l=1}^N(\lambda\beta_l+\alpha)+2\alpha)L\) [for Standard TWFPD Synthesis]
\end{itemize}
Thus the complexity in this case is given as
\(
(\sum_{l=1}^N\beta_l/\lambda^\dm+(N+5)\alpha+\lambda\sum_{l=1}^N\beta_l)L,
\)
and by considering the average \(\beta^\ast:=(\sum_{l=1}^N\beta_l)/N\) and using \(N\le \lambda^\dm\), we see that the complexity is bounded by
\[
\left((N+5)\alpha+(\lambda N+1)\beta^\ast\right)L.
\]
Therefore, the standard fast TWFPD algorithms have linear complexity with complexity constant \((N+5)\alpha+(\lambda N+1)\beta^\ast\). The complexity constant in this case increases as the number of directions \(N\) increases. 

\medskip\noindent{\bf Complexity with LP-based Fast TWFPD Synthesis Algorithm.}
Let \(L\), \(\alpha\) and \(\beta_l\) be defined as before.
Since the analysis algorithm in this case is exactly the same as before, the number of operations required in a complete cycle of \(1\)-level-down analysis and \(1\)-level-up synthesis with LP-based TWFPD synthesis algorithm is given as the sum of the following numbers: 
\begin{itemize}
\item \(\alpha L+(\sum_{l=1}^N\beta_l/\lambda^\dm)L+\alpha L\) \; [for Standard TWFPD Analysis]
\item  \(\alpha L\) \quad  \quad \quad \quad \quad \quad  \quad \quad \quad \quad \quad [for step {\tt (iv)} in LP-based TWFPD Synthesis]
\end{itemize}
The complexity in this case is 
\(
(\sum_{l=1}^N\beta_l/\lambda^\dm+3\alpha)L,
\)
and from \(\beta^\ast=(\sum_{l=1}^N\beta_l)/N\) and \(N\le \lambda^\dm\) as before, the complexity is bounded by
\be
\label{eq:complexity}
\left(3\alpha+\beta^\ast\right)L.
\ee
As is expected, the complexity constant is much smaller in this case; hence, these fast TWFPD algorithms with  LP-based TWFPD synthesis algorithm are much faster. The algorithms have linear complexity with complexity constant \(3\alpha+\beta^\ast\), which stays the same even if the number of directions \(N\) gets higher. 

\section{Examples}
\label{S:Examples}

In this section, we present some examples to illustrate our method of constructing TWFPD in Theorem~\ref{thm:main}. Although our construction method works for any dimension \(\dm\ge 2\), we consider the case when \(\dm=2\) for simplicity in all of our examples with only one exception. In Example~\ref{ex:exone}, we begin the construction with \(\dm= 2\) case and then generalize it to \(\dm\ge 2\) cases.

\begin{example}
\label{ex:exone}
Let \(\dm=2\), and consider \(N=3\) directions in \(\ZZ^2\). Assume \(\lambda=2\). Let the vanishing moment vector be \((1,1,1)\), and let the direction matrix be
\[
[\xi_1, \xi_2, \xi_3]
=\left[
\begin{array}{ccc}
1&0&1\\ 
0&1&1
\end{array}
\right].
\]
Then, by (\ref{eq:defgl}), we have for \(1\le l\le 3\),
\[
g_l(\omega)=(1-e^{-i\xi_l\cdot\omega})/4, \quad \omega\in\TT^2.
\]
For \(1\le l\le 3\), let
\[
p_l(\omega)=(1+e^{-i\xi_l\cdot\omega})/2,\quad \omega\in\TT^2, 
\]
and let \(\nu_l=\xi_l\), for \(1\le l\le 3\), and \(\nu_4=0\). Then the lowpass mask of the TWFPD in this case is given as, for \(\omega=(\omega_1,\omega_2)\in\TT^2\)
\[\tau(\omega)=1/2+(e^{i\omega_1}+e^{-i\omega_1})/4+(e^{i\omega_2}+e^{-i\omega_2})/4+(e^{i(\omega_1+\omega_2)}+e^{-i(\omega_1+\omega_2)})/4.\]
We note the refinable function associated with this lowpass mask \(\tau\) is the 2-D piecewise linear box spline. From this, we know that both the accuracy number and the flatness number of \(\tau\) are equal to 2. The lowpass filter \(h\) associated with \(\tau\) is depicted in Fig.~\ref{fig:Ex1_hANDhl}(a)\footnote{When drawing a filter in this paper, we use a box to indicate its value  at the origin.}, and the highpass filters \(h_l\), \(l=1,2,3\) associated with \(g_l\), \(l=1,2,3\) (cf. (\ref{eq:deftauh})) are depicted in Fig.~\ref{fig:Ex1_hANDhl}(b)-(d)\footnote{When drawing a highpass filter in this paper, its constant multiple is shown for clear directionality.}.

\begin{figure}[t]
\centering
\begin{subfigure}[b]{0.8\textwidth}
\centering
\begin{tabular}{ccc}
0&1/4&1/4\\
1/4&\fbox{1/2}&1/4\\
1/4&1/4&0
\end{tabular}
\caption{\(h\) }
\end{subfigure}\unskip \
\begin{subfigure}[b]{0.3\textwidth}
\vspace{1mm}
\centering
\begin{tabular}{cc}
0&0\\
\fbox{1}&-1
\end{tabular}
\caption{\(4\times\) (\(h_1\))}
\end{subfigure}\unskip \
\begin{subfigure}[b]{0.3\textwidth}
\vspace{1mm}
\centering
\begin{tabular}{cc}
-1&0\\
\fbox{1}&0
\end{tabular}
\caption{\(4\times\) (\(h_2\))}
\end{subfigure}\unskip \
\begin{subfigure}[b]{0.3\textwidth}
\vspace{1mm}
\centering
\begin{tabular}{cc}
0&-1\\
\fbox{1}&0
\end{tabular}
\caption{\(4\times\) (\(h_3\))}
\end{subfigure}
\caption{Lowpass filter \(h\) associated with \(\tau\), and highpass filters \(h_l\) associated with \(g_l\), for \(l=1,2,3\)  in Example~\ref{ex:exone} (\(\dm=2\)).}
\label{fig:Ex1_hANDhl}
\end{figure}

The directional wavelet masks in this case are, for \(\omega\in\TT^2\),
\[
q_{D,l}(\omega)=\tau(\omega)(1-e^{-2i\xi_l\cdot\omega})/4,\quad 1\le l\le 3
\]
with the associated highpass filters depicted in Fig.~\ref{fig:Ex1_waveletfilters}(a)-(c), 
and the complementary wavelet masks are 
\[
q_{C,\mu}(\omega)=
\begin{cases} 
e^{i\xi_\mu\cdot\omega}-\tau(\omega)(1+e^{2i\xi_\mu\cdot\omega})/4&1\le \mu\le 3,\\
1- \tau(\omega)/2&\mu=4,
\end{cases}
\]
with the associated highpass filters depicted in Fig.~\ref{fig:Ex1_waveletfilters}(d)-(g). 
By Corollary~\ref{coro:VM}, each \(q_{D,l}\) has exactly \(1\) vanishing moment, and each \(q_{C,\mu}\) has at least two vanishing moments. Direct computation shows that each \(q_{C,\mu}\) has exactly two vanishing moments. 

The TWFPD, in this case, is the same as the tight wavelet filter bank constructed in Example 1~of Ref.~\refcite{HL1} in dimension 2, up to shifts of filters. 

\begin{figure}[t]
\centering
\begin{subfigure}[b]{0.3\textwidth}
\centering
\begin{tabular}{ccccc}
&1&1&-1&-1\\
1&\fbox{2}&0&-2&-1\\
1&1& -1 & -1& 
\end{tabular}
\caption{\(8\times\) (\(q_{D,1}\) filter)}
\end{subfigure}\unskip \
\begin{subfigure}[b]{0.3\textwidth}
\vspace{1mm}
\centering
\begin{tabular}{ccc}
&-1&-1\\
-1&-2&-1\\
-1&0&1\\
1& \fbox{2}&1\\
1&1& 
\end{tabular}
\caption{\(8\times\) (\(q_{D,2}\) filter)}
\end{subfigure}\unskip \
\begin{subfigure}[b]{0.3\textwidth}
\vspace{1mm}
\centering
\begin{tabular}{ccccc}
&&&-1&-1\\
&&-1&-2&-1\\
&1&0&-1&\\
1&\fbox{2}&1&&\\
1&1&&&
\end{tabular}
\caption{\(8\times\) (\(q_{D,3}\) filter)}
\end{subfigure}\unskip \
\begin{subfigure}[b]{0.3\textwidth}
\vspace{1mm}
\centering
\begin{tabular}{ccccc}
&-1&-1&-1&-1\\
-1&-2&14&\fbox{-2}&-1\\
-1&-1&-1&-1&
\end{tabular}
\caption{\(16\times\) (\(q_{C,1}\) filter)}
\end{subfigure}
\begin{subfigure}[b]{0.3\textwidth}
\vspace{1mm}
\centering
\begin{tabular}{ccc}
&-1&-1\\
-1&\fbox{-2}&-1\\
-1&14&-1\\
-1&-2&-1\\
-1&-1&
\end{tabular}
\caption{\(16\times\) (\(q_{C,2}\) filter)}
\end{subfigure}
\begin{subfigure}[b]{0.3\textwidth}
\vspace{1mm}
\centering
\begin{tabular}{ccccc}
&&&-1&-1\\
&&-1&\fbox{-2}&-1\\
&-1&14&-1&\\
-1&-2&-1\\
-1&-1&
\end{tabular}
\caption{\(16\times\) (\(q_{C,3}\) filter)}
\end{subfigure}
\begin{subfigure}[b]{0.3\textwidth}
\vspace{1mm}
\centering
\begin{tabular}{ccc}
&-1&-1\\
-1&\fbox{6}&-1\\
-1&-1&
\end{tabular}
\caption{\(8\times\) (\(q_{C,4}\) filter)}
\end{subfigure}
\caption{Highpass filters associated with directional masks \(q_{D,l}\), \(1\le l\le 3\) ((a)-(c)) and with complementary masks \(q_{C,\mu}\), \(1\le \mu\le 4\) ((d)-(g)) in Example~\ref{ex:exone} (\(\dm=2\)).}
\label{fig:Ex1_waveletfilters}
\end{figure}

In fact, our TWFPD construction and the tight wavelet filter bank construction in Example 1~of Ref.~\refcite{HL1} work for any dimension \(\dm\ge 2\), and they give exactly the same tight wavelet filter banks, up to shifts of filters. To see this, for dimension \(\dm\ge 2\), consider \(N=2^\dm-1\) directions in \(\Zd\). Suppose that \(\lambda=2\), and the vanishing moment vector is the vector of \(1\)'s. Let the directionality vectors \(\xi_1,\dots,\xi_{2^\dm-1}\) be the elements of \(\{0,1\}^\dm\setminus\{0\}\), ordered in some fixed way.
Then, for \(1\le l\le 2^\dm-1\),
\[
g_l(\omega)=(1-e^{-i\xi_l\cdot\omega})/2^{\dm/2+1},\quad\omega\in\Td.
\]
Define, similar to the case of dimension \(2\), for \(1\le l\le 2^\dm-1\)
\[
p_l(\omega)=(1+e^{-i\xi_l\cdot\omega})/2,\quad\omega\in\Td.
\]
Let \(\nu_l=\xi_l\), for \(1\le l\le 2^\dm-1\), and \(\nu_{2^\dm}=0\). Then the lowpass mask \(\tau\) is
\[
\tau(\omega)=1/{2^{\dm/2}}+\sum_{l=1}^{2^\dm-1}(e^{i\xi_l\cdot\omega}+e^{-i\xi_l\cdot\omega})/{2^{\dm/2+1}}, \quad\omega\in\Td,
\]
which has the \(\dm\)-D piecewise linear box spline as the associated refinable function. The directional wavelet masks of TWFPD are, for \(\omega\in\Td\)
\[
q_{D,l}(\omega)=\tau(\omega)(1-e^{-2i\xi_l\cdot\omega})/2^{\dm/2+1},\quad 1\le l\le 2^\dm-1
\]
and the complementary wavelet masks of TWFPD are 
\[
q_{C,\mu}(\omega)=
\begin{cases} 
e^{i\xi_\mu\cdot\omega}-\tau(\omega)(1+e^{2i\xi_\mu\cdot\omega})/4&1\le \mu\le 2^\dm-1,\\
1-\tau(\omega)/2&\mu=2^\dm,
\end{cases}
\]
Similar to the 2-D case, we see that each \(q_{D,l}\) has exactly one vanishing moment, and each \(q_{C,\mu}\) has exactly two vanishing moments.

For complexity computation, note that the number of nonzero entries in the lowpass filter \(h\) of this TWFPD is \(\alpha=2^{\dm+1}-1\) and the average number of nonzero entries in the highpass filters \(h_l\) (associated with \(g_l\)), \(1\le l\le 2^\dm-1\), of this TWFPD is \(\beta^\ast=2\) (cf. Section~\ref{subS:Algorithms}). Hence the fast TWFPD algorithms with LP-based TWFPD synthesis algorithm have linear complexity with complexity constant \(3\alpha+\beta^\ast=6\cdot 2^{\dm}-1\) (cf. (\ref{eq:complexity})). In particular, when \(\dm=2\), the complexity constant is \(23\). For a comparison, the fast algorithms of \(\dm\)-D tensor-product based Haar orthonormal wavelets have linear complexity with complexity constant \(4\dm\cdot 2^{\dm}\), which is 32 when \(\dm=2\).\; \qed
\end{example}

\begin{figure}[t]
\centering
\begin{subfigure}[b]{0.8\textwidth}
\centering
\begin{tabular}{cccc}
0&1/4&0&0\\
0&1/4&1/4&0\\
1/4&\fbox{0}&1/4&1/4\\
1/4&1/4&0&0
\end{tabular}
\caption{\(h\) }
\end{subfigure}\unskip \
\begin{subfigure}[b]{0.2\textwidth}
\vspace{1mm}
\centering
\begin{tabular}{cc}
0&0\\
\fbox{1}&-1
\end{tabular}
\caption{\(4\times\) (\(h_1\))}
\end{subfigure}\unskip \
\begin{subfigure}[b]{0.2\textwidth}
\vspace{1mm}
\centering
\begin{tabular}{cc}
-1&0\\
\fbox{1}&0
\end{tabular}
\caption{\(4\times\) (\(h_2\))}
\end{subfigure}\unskip \
\begin{subfigure}[b]{0.2\textwidth}
\vspace{1mm}
\centering
\begin{tabular}{cc}
0&-1\\
\fbox{1}&0
\end{tabular}
\caption{\(4\times\) (\(h_3\))}
\end{subfigure}
\begin{subfigure}[b]{0.2\textwidth}
\vspace{1mm}
\centering
\begin{tabular}{cc}
-1&0\\
\fbox{0}&1
\end{tabular}
\caption{\(4\times\) (\(h_4\))}
\end{subfigure}
\caption{Lowpass filter \(h\) and highpass filters \(h_l\), \(l=1,2,3,4\), in Example~\ref{ex:extwo}.}
\label{fig:Ex2_hANDhl}
\end{figure}

\begin{example}
\label{ex:extwo}
Consider adding one more direction to the directions considered in Example~\ref{ex:exone} for the 2-D case. 
Suppose that \(N=4\), \(\lambda=2\), and the vanishing moment vector is the vector of \(1\)'s. Let the direction matrix be 
\[
[\xi_1, \xi_2, \xi_3, \xi_4]=\left[
\begin{array}{rrrr}
1&0&1&-1\\ 
0&1&1&1
\end{array}
\right].
\]
Let the initial point of direction \(\xi_4\) be \(\zeta_4=(1,0)\) and the other initial points be zero. Then, for \(\omega=(\omega_1,\omega_2)\in\TT^2\), by (\ref{eq:defgl}) and (\ref{eq:defglg}),
\[
g_l(\omega)=(1-e^{-i\xi_l\cdot\omega})/4,\quad 1\le l\le 3,\quad g_4(\omega)=e^{-i\omega_1}(1-e^{i\xi_4\cdot\omega})/4.
\]
Let
\[
p_l(\omega)=(1+e^{-i\xi_l\cdot\omega})/2, \quad 1\le l\le 4,
\]
let \(\nu_l=\xi_l\) for \(1\le l\le 3\), and let
\[
\nu_4=
\left[\begin{array}{c}
-2\\ 
0
\end{array}\right].
\]
Then the lowpass mask of this TWFPD is \(\)
\beaN
\tau(\omega)&{\;=\;}&(e^{i\omega_1}+e^{-i\omega_1})/4+(e^{i\omega_2}+e^{-i\omega_2})/4\\
&+&(e^{i(\omega_1+\omega_2)}+e^{-i(\omega_1+\omega_2)})/4+(e^{-2i\omega_1}+e^{-2i\omega_2})/4,\quad \omega=(\omega_1,\omega_2)\in\TT^2.
\eeaN
Fig.~\ref{fig:Ex2_hANDhl} shows the associated lowpass filter \(h\), and the highpass filters \(h_l\), \(l=1,2,3,4\). It is easy to check that the accuracy number and the flatness number of \(\tau\) are exactly one. This, together with the fact that \(m_l=1\), \(1\le l\le 4\), gives that every wavelet mask of this TWFPD has exactly one vanishing moment (cf. Corollary~\ref{coro:VM}).

In this case, \(\alpha=8\) and \(\beta^\ast=2\), hence the fast TWFPD algorithms with LP-based TWFPD synthesis algorithm have linear complexity with complexity constant \(26\).

The directions considered in this TWFPD are the same as the directions studied in Ref.~\refcite{LCSHT}, where the directional tight wavelet filter bank with 2-D Haar lowpass mask is constructed, but the construction there does not allow one to choose the directions a priori.\ \qed
\end{example}

\begin{figure}[t]
\centering
\begin{subfigure}[b]{0.8\textwidth}
\centering
\begin{tabular}{ccccccccc}
1/6&0&0&0&0&1/6&0&0\\
0&1/6&0&0&0&0&0&0\\
0&0&0&1/6&0&1/6&0&1/6\\
0&0&0&0&1/6&0&0&0\\
0&0&1/6&\fbox{1/3}&0&1/6&0&1/6\\
0&1/6&1/6&1/6&0&0&0&0\\
0&0&1/6&1/6&0&0&0&1/6
\end{tabular}
\caption{\(h\) }
\end{subfigure}\unskip \
\begin{subfigure}[b]{0.2\textwidth}
\vspace{1mm}
\centering
\begin{tabular}{cc}
0&0\\
\fbox{1}&-1
\end{tabular}
\caption{\(6\times\) (\(h_1\))}
\end{subfigure}\unskip \
\begin{subfigure}[b]{0.2\textwidth}
\vspace{1mm}
\centering
\begin{tabular}{cc}
-1&0\\
\fbox{1}&0
\end{tabular}
\caption{\(6\times\) (\(h_2\))}
\end{subfigure}\unskip \
\begin{subfigure}[b]{0.2\textwidth}
\vspace{1mm}
\centering
\begin{tabular}{cc}
0&-1\\
\fbox{1}&0
\end{tabular}
\caption{\(6\times\) (\(h_3\))}
\end{subfigure}
\begin{subfigure}[b]{0.2\textwidth}
\vspace{1mm}
\centering
\begin{tabular}{cc}
-1&0\\
\fbox{0}&1
\end{tabular}
\caption{\(6\times\) (\(h_4\))}
\end{subfigure}
\begin{subfigure}[b]{0.2\textwidth}
\vspace{1mm}
\centering
\begin{tabular}{ccc}
0&0&-1\\
\fbox{1}&0&0
\end{tabular}
\caption{\(6\times\) (\(h_5\))}
\end{subfigure}
\begin{subfigure}[b]{0.2\textwidth}
\vspace{1mm}
\centering
\begin{tabular}{cc}
0&-1\\
0&0\\
\fbox{1}&0
\end{tabular}
\caption{\(6\times\) (\(h_6\))}
\end{subfigure}
\begin{subfigure}[b]{0.2\textwidth}
\vspace{1mm}
\centering
\begin{tabular}{cc}
-1&0\\
0&0\\
\fbox{0}&1
\end{tabular}
\caption{\(6\times\) (\(h_7\))}
\end{subfigure}
\begin{subfigure}[b]{0.2\textwidth}
\vspace{1mm}
\centering
\begin{tabular}{ccc}
-1&0&0\\
\fbox{0}&0&1
\end{tabular}
\caption{\(6\times\) (\(h_8\))}
\end{subfigure}
\caption{Lowpass filter \(h\) and highpass filters \(h_l\), \(1\le l\le 8\), in Example~\ref{ex:exthree}.}
\label{fig:Ex3_hANDhl}
\end{figure}

\begin{example}
\label{ex:exthree}
Consider adding four more directions to the directions considered in Example~\ref{ex:extwo}, and let  \(N=8\). In this case \(\lambda=2\) cannot be used as it does not satisfy the constraint \(N\le \lambda^2\). We choose \(\lambda=3\) and the vanishing moment vector to be the vector of \(1\)'s. Let the direction matrix be 
\[
[\xi_1, \xi_2, \xi_3, \xi_4,\xi_5,\xi_6,\xi_7,\xi_8]=\left[
\begin{array}{rrrrrrrrr}
1&0&1&-1&2&1&-1&-2\\ 
0&1&1&1&1&2&2&1
\end{array}
\right].
\]
Let the initial point of directions \(\xi_4\), \(\xi_7\), and \(\xi_8\) be \(\zeta_4=\zeta_7=(1,0)\), and \(\zeta_8=(2,0)\), respectively, and let the other initial points be zero. Then, for \(\omega=(\omega_1,\omega_2)\in\TT^2\), by (\ref{eq:defgl}) and (\ref{eq:defglg}), we have \(g_8(\omega)=e^{-2i\omega_1}(1-e^{i\xi_8\cdot\omega})/6\), and
\[
g_l(\omega)=(1-e^{-i\xi_l\cdot\omega})/6,\quad l=1,2,3,5,6,
\]
\[
g_l(\omega)=e^{-i\omega_1}(1-e^{i\xi_l\cdot\omega})/6,\quad l=4,7.
\]
By proceeding similarly as in Example~\ref{ex:exone} and \ref{ex:extwo}, let 
\[
p_l(\omega)=(1+e^{-i\xi_l\cdot\omega})/2,\quad 1\le l\le 8, \quad \omega\in\TT^2.
\]
Let \(\nu_l=\xi_l\) for \(l=1,2,3,5,6\), let \(\nu_9=0\), and let
\[\nu_4=
\left[\begin{array}{c}
-4\\ 
2
\end{array}\right],\quad
\nu_7=
\left[\begin{array}{c}
0\\ 
2
\end{array}\right],\quad
\nu_8=
\left[\begin{array}{c}
-4\\ 
0
\end{array}\right].
\] 
Then the lowpass mask is given as, for \(\omega=(\omega_1,\omega_2)\in\TT^2\)
\beaN
\tau(\omega)&{\;=\;}&1/3+\sum_{l\in\{1,2,3,5,6\}}(e^{i\xi_l\cdot\omega}+e^{-2i\xi_l\cdot\omega})/6+(e^{-4i\omega_1}e^{2i\omega_2}+e^{-i\omega_1}e^{-i\omega_2})/6\\
&+&(e^{2i\omega_2}+e^{3i\omega_1}e^{-4i\omega_2})/6
+(e^{-4i\omega_1}+e^{2i\omega_1}e^{-3i\omega_2})/6.
\eeaN
The lowpass filter \(h\) associated with \(\tau\) and the highpass filters \(h_l\) associated with \(g_l\), for \(1\le l\le 8\) are depicted in Fig.~\ref{fig:Ex3_hANDhl}. As in Example~\ref{ex:extwo}, the accuracy number and the flatness number of \(\tau\) of this TWFPD are exactly one, and every wavelet mask of this TWFPD has exactly one vanishing moment. 

Since \(\alpha=17\) and \(\beta^\ast=2\) in this example, the complexity constant for linear fast TWFPD algorithms with LP-based TWFPD synthesis algorithm is \(53\).\ \qed
\end{example}

\begin{figure}[t]
\centering
\begin{subfigure}[b]{0.8\textwidth}
\centering
\begin{tabular}{ccccc}
0&\(\frac{1}{8}(1-\sqrt{2})\)&0&0&\(\frac{1}{8}(1-\sqrt{2})\)\\
0&0&0&0&0\\
0&1/4&1/4&0&0\\
\(\frac{1}{8}(1+\sqrt{2})\)&\fbox{1/2}&1/4&0&\(\frac{1}{8}(1-\sqrt{2})\)\\
\(\frac{1}{8}(1+\sqrt{2})\)&\(\frac{1}{8}(1+\sqrt{2})\)&0&0&0
\end{tabular}
\caption{\(h\) }
\end{subfigure}\unskip \
\begin{subfigure}[b]{0.3\textwidth}
\vspace{1mm}
\centering
\begin{tabular}{ccc}
0&0&0\\
\fbox{1}&-2&1
\end{tabular}
\caption{\(8\times\) (\(h_1\))}
\end{subfigure}\unskip \
\begin{subfigure}[b]{0.3\textwidth}
\vspace{1mm}
\centering
\begin{tabular}{cc}
1&0\\
-2&0\\
\fbox{1}&0
\end{tabular}
\caption{\(8\times\) (\(h_2\))}
\end{subfigure}\unskip \
\begin{subfigure}[b]{0.3\textwidth}
\vspace{1mm}
\centering
\begin{tabular}{ccc}
0&0&1\\
0&-2&0\\
\fbox{1}&0&0
\end{tabular}
\caption{\(8\times\) (\(h_3\))}
\end{subfigure}
\caption{Lowpass filter \(h\) and highpass filters \(h_l\), \(l=1,2,3\), in Example~\ref{ex:exfour}.}
\label{fig:Ex4_hANDhl}
\end{figure}

\begin{example}
\label{ex:exfour}
Keeping precisely the same directions as in Example~\ref{ex:exone} for the 2-D case, consider increasing the number of vanishing moments along each direction from 1 to 2. For this, let \(N=3\) and \(\lambda=2\), let the direction matrix \([\xi_1, \xi_2,\xi_3]\) be the same as in Example~\ref{ex:exone}, and let the vanishing moment vector be \((2,2,2)\).
Then, for \(\omega=(\omega_1,\omega_2)\in\TT^2\)
\[
g_l(\omega)=(1-2e^{-i\xi_l\cdot\omega}+e^{-2i\xi_l\cdot\omega})/8,\quad 1\le l\le 3.
\]
Let \(\nu_l=\xi_l\), for \(1\le l\le 3\), \(\nu_4=0\), and
\[
p_l(\omega)=\left(1+\sqrt{2}+2e^{-i\xi_l\cdot\omega}+(1-\sqrt{2})e^{-2i\xi_l\cdot\omega}\right)/4,\quad 1\le l\le 3.
\]
Then, for \(\omega=(\omega_1,\omega_2)\in\TT^2\), 
\beaN
\tau(\omega)&{\;=\;}&\left((1+\sqrt{2})e^{i\omega_1}+2e^{-i\omega_1}+(1-\sqrt{2})e^{-3i\omega_1}\right)/8+1/2\\
&+&\left((1+\sqrt{2})e^{i\omega_2}+2e^{-i\omega_2}+(1-\sqrt{2})e^{-3i\omega_2}\right)/8\\
&+&\left((1+\sqrt{2})e^{i(\omega_1+\omega_2)}+2e^{-i(\omega_1+\omega_2)}+(1-\sqrt{2})e^{-3i(\omega_1+\omega_2)}\right)/8.
\eeaN
The lowpass filter \(h\) and the highpass filters \(h_l\), \(l=1,2,3\) are drawn in Fig.~\ref{fig:Ex4_hANDhl}. For this TWFPD, each directional wavelet mask has exactly two vanishing moments and each complementary wavelet mask has exactly one vanishing moment. In the fast TWFPD algorithms with LP-based synthesis using this TWFPD, \(\alpha=10\), \(\beta^\ast=3\), and the complexity constant is \(33\).\ \qed
\end{example}

\section*{Acknowledgments}
This work was supported in part by the National Research Foundation of Korea (NRF) grant [No. 2015R1A5A1009350 and No. 2021R1A2C1007598] and in part by the Institute of Information \& Communications Technology Planning \& Evaluation (IITP) grant [No. 2021-0-00023], both funded by the Korea government (MSIT).

\bibliographystyle{acm} 
\bibliography{IEEEabrv,BibliographyAll}

\end{document}